\documentclass[12pt,twoside]{amsart}
\usepackage{amsmath}
\usepackage{amssymb}
\usepackage{amscd}
\usepackage{xypic}
\xyoption{all}
\title[Mixed Hodge structure associated to hypersurface singularities]{Mixed Hodge structure associated to hypersurface singularities}

\author{Mohammad Reza Rahmati}
\thanks{}
\address{Centro de Investigacion en Matematicas , A.C.
\hfill\break 
\hfill\break \\
\hfill\break }
\email{mrahmati@cimat.mx}

\newcommand{\comments}[1]{}

\newtheorem{theorem}{Theorem}[section]
\newtheorem{proposition}[theorem]{Proposition}
\newtheorem{corollary}[theorem]{Corollary}

\newtheorem{remark}[theorem]{Remark}
\newtheorem{example}[theorem]{Example}



\begin{document}

\begin{abstract}
Let $f:\mathbb{C}^{n+1} \to \mathbb{C}$ be a germ of hypersurface with isolated singularity. One can associate to $f$ a polarized variation of mixed Hodge structure $\mathcal{H}$ over the punctured disc, where the Hodge filtration is the limit Hodge filtration of W. Schmid and J. Steenbrink. By the work of M. Saito and P. Deligne the VMHS associated to cohomologies of the fibers of $f$ can be extended over the degenerate point $0$ of disc. The new fiber obtained in this way is isomorphic to the module of relative differentials of $f$ denoted $\Omega_f$. A mixed Hodge structure can be defined on $\Omega_f$ in this way. The polarization on $\mathcal{H}$ deforms to Grothendieck residue pairing modified by a varying sign on the Hodge graded pieces in this process. This also proves the existence of a Riemann-Hodge bilinear relation for Grothendieck pairing and allow to calculate the Hodge  signature of Grothendieck pairing.  
\end{abstract}

\maketitle


\section*{Introduction}

\vspace{0.5cm}

One of the important subjects of study in Hodge theory and D-modules is the asymptotic behaviour of the underlying variation of (mixed) Hodge structures and also the possibility to extend such structures over the degeneracy. If the VMHS is coming from an isolated hypersurface singularity, the mixed Hodge structure we are interested in is the Steenbrink limit mixed Hodge structure (W. Schmid LMHS in projective fibrations). Classically there are two equivalent ways to define Steenbrink LMHS. The first, due to J. Steenbrink is by a spectral sequence argument applied to the resolution of singularities in a projective fibration. In the affine case one needs to embed the original family in a projective one and then use the Invariant cycle theorem. Another method is by the structure of lattices in the Gauss-Manin system associated to VMHS on the punctured disc. It concerns the definition of Brieskorn lattices. By applying the powers of the Gauss-Manin connection to the Brieskorn lattice one can define a decreasing filtration, that is equivalent to the limit Hodge filtration of J. Steenbrink. Our strategy is to extend such local system of MHS over the degeneracy. We explain this by gluing vector bundles. It mainly involves the fact that, $V$-lattices can be glued with the Brieskorn lattices in different charts. An observation made by P. Deligne and M. Saito is that, one can define a new filtration on the vanishing cohomology, that is opposite to the original Hodge filtration in this way. 

Throughout the whole text assume $f:\mathbb{C}^{n+1} \to \mathbb{C}$ is a holomorphic germ with isolated singularity at $0 \in \mathbb{C}^{n+1}$ and $f:X \to T$ a representative for the Milnor fibration. One can associate a local system to this data by

\vspace{0.2cm}

\begin{center}
$\mathcal{H}=R^nf_*\mathbb{C}= \bigcup_{t \in T'}H^n(X_t,\mathbb{C})$
\end{center}

\vspace{0.2cm}

Denote by $M$ the monodromy, and $N=\log(M_u) \otimes 1/2 \pi i$ the logarithm of the unipotent part of $M$. By the monodromy theorem $M$ is quasi-unipotent. We will use the notation $X_{\infty}=H \times_{T'} X_t$, where $H$ is the upper half plane for the canonical fiber, and

\vspace{0.2cm}
 
\begin{center}
$\Omega_f \cong \displaystyle{\frac{\Omega_{X,0}^{n+1}}{df \wedge \Omega_{X,0}^{n}}}$ 
\end{center}

\vspace{0.2cm}

Let $\mathcal{G}$ be the associated Gauss-Manin system via the Riemann-Hilbert correspondence. We shall extend the Gauss-Manin system $\mathcal{G}$ as well as $\mathcal{H}$ over the puncture. The stack on $0$ of the extended Gauss-Manin module can be naturally identified with the module of relative differentials $\Omega_f$. This is a direct consequence of isomorphisms between $V$-lattices and the Brieskorn lattices, which also pairs the Hodge filtration and the (actually a re-indexed) $V$-filtration in an opposite way. A classical example of this situation is when the variation of MHS is Hodge-Tate. We show a modification of Grothendieck residue defines a polarization for a mixed Hodge structure on $\Omega_f$.

\begin{theorem}
Assume $f:\mathbb{C}^{n+1} \to \mathbb{C}$ is a holomorphic hypersurface germ with isolated singularity at $0 \in \mathbb{C}^{n+1}$. Then the variation of mixed Hodge structure, cf. \cite{H1}. This VMHS can be extended to the puncture with the extended fiber isomorphic to $\Omega_f$ in the sense of sec. 4 and 5, and it is polarized by 5.1. The Hodge filtration on the new fiber $\Omega_f$ corresponds to an opposite Hodge filtration on $H^n(X_{\infty},\mathbb{C})$ in the way explained in Sec. 4.
\end{theorem}

The way to reach to this result is by a standard gluing argument on the structure theory of Gauss-Manin system explained in \cite{SCHU}. We equip $\Omega_f$ with a mixed Hodge structure compatible with this gluing. The definition is based on a linear isomorphism $\Phi$ between the two vector spaces $(H^n(X_{\infty},\mathbb{C}),M)$ and $\Omega_f$, which already have been used by A. Varchenko in its graded form. However we take a lift of the map employed by Varchenko. The main theorem 5.1 express that this isomorphism pairs the polarization on the vanishing cohomology and a modification of the Grothendieck residue of $f$. The computation of a twist in Grothendieck residue which goes back to the inter-relation with higher residue pairing with the linking form of $f$ makes the contribution non-trivial. It also allows to formulate a Riemann-Hodge bilinear relation for Grothendieck residue. 
We present several applications of the above procedure.

\vspace{0.3cm}

\section{Steenbrink limit Mixed Hodge structure}

\vspace{0.5cm}

Suppose we have an isolated singularity holomorphic germ $f:\mathbb{C}^{n+1} \to \mathbb{C}$. By the Milnor fibration Theorem we can always associate to $f$ a $C^{\infty}$-fiber bundle over a small punctured disc $T^{\prime}$. The corresponding cohomology bundle $\mathcal{H}$, constructed from the middle cohomologies of the fibers defines a variation of mixed Hodge structure on $T^{\prime}$. Recall the Brieskorn lattice is defined by,

\vspace{0.2cm}

\begin{center}
$(\mathcal{H}^{(0)}=)H^{\prime \prime}=f_*\displaystyle{\frac{\Omega_{X,0}^{n+1}}{df \wedge d\Omega_{X,0}^{n-1}}}$
\end{center}

\vspace{0.2cm}

The Brieskorn lattice is the stack at $0$ of a locally free $\mathcal{O}_T$-module $\mathcal{H}^{\prime \prime}$ of rank $\mu$ with  $\mathcal{H}^{\prime \prime}_{T'} \cong \mathcal{H}$, and hence $H^{\prime \prime} \subset (i_*\mathcal{H})_0$, with $i:T^{\prime} \hookrightarrow T$. The regularity of the Gauss-Manin connection proved by Brieskorn 
and Malgrange implies that $H^{\prime \prime} \subset \mathcal{G}$. It is a theorem by Malgrange (see \cite{SCHU} 1.4.10) $H^{\prime \prime} \subset V^{-1}$. The Leray residue formula can be used to express the action of $\partial_t$ in terms of differential forms by $\partial_t^{-1}s[d\omega]= s[df \wedge \omega]$, where, $\omega \in \Omega_X^n$ and $s(\eta)(t)=Res_{f=t}(\eta/(f-t))$. In particular, $\partial_t^{-1}.H^{\prime \prime} \subset H^{\prime \prime}$, and
 
\begin{equation}
\displaystyle{\frac{H^{\prime \prime}}{\partial_t^{-1}.H^{\prime \prime}}} \cong \displaystyle{\frac{\Omega_{X,0}^{n+1}}{df \wedge \Omega_{X,0}^{n}}} \cong  \displaystyle{\frac{\mathbb{C}\{\underline{z}\}}{( \partial(f))}}.
\end{equation}

\vspace{0.3cm}

\noindent
for the module of relative differentials of the map $f$. 
The Hodge filtration on $H^n(X_{\infty},\mathbb{C})$ is defined by

\begin{equation}
F^pH^n(X_{\infty})_{\lambda}=\psi_{\alpha}^{-1}\partial_t^{n-p} Gr_V^{\alpha +n-p}\mathcal{H}^{(0)}.
\end{equation}

\vspace{0.2cm}

\noindent
where $\psi_{\alpha}:H^n(X_{\infty})_{\lambda} \to C^{\alpha} \subset \mathcal{G}$ is 

\vspace{0.2cm}

\begin{center}
$\psi_{\alpha}(A_{\alpha})=t^{\alpha}\exp(\log(t)N/2\pi i)A_{\alpha}, \qquad \lambda=\exp(2\pi i \alpha ), \ -1 \leq \alpha <0$
\end{center}

\vspace{0.2cm}

\noindent
Therefore, $Gr_F^{p}H^n(X_{\infty},\mathbb{C})_{\lambda}=Gr_V^{\alpha+n-p}\Omega_f$. The $V$-filtration on $\Omega_f$ is defined by

\vspace{0.2cm}

\begin{equation}
V^{\alpha} \Omega_f= pr(V^{\alpha} \cap H'')
\end{equation}

\vspace{0.2cm}

\noindent
Clearly $V^{\alpha}\Omega_f=\oplus_{\beta \geq \alpha} \Omega_f^{\beta}$ and $\Omega_f \cong \oplus Gr_V^{\alpha}\Omega_f$ hold. 

A theorem of A. Varchenko, relates the operator $N$, on vanishing cohomology and multiplication by $f$ on $\Omega_f$. More specifically, the maps $Gr(f)$ and $N=\log{M_u} \in End (H^n(X_{\infty},\mathbb{C}))$ have the same Jordan normal forms. We have a canonical isomorphism 

\vspace{0.3cm}

\begin{center}
$Gr_FH^n(X_{\infty}, \mathbb{C})= \displaystyle{\bigoplus_{-1<\alpha \leq 0}}Gr_FC^{\alpha}$
\end{center}

\vspace{0.3cm}

\noindent
and the corresponding endomorphism 

\vspace{0.2cm}

\begin{center}
$N_{p,\alpha}:Gr_F^pC^{\alpha} \to Gr_F^{p-1}C^{\alpha}$ \\[0.2cm] $\qquad \qquad \qquad N_{p,\alpha}(x)=-2 \pi i (t\partial_t-\alpha)x \cong -2 \pi i .t \partial_t x \qquad (\mod F^p)$.
\end{center}

\vspace{0.2cm}

\noindent
If $\beta \in \mathbb{Q},\ \beta=n-p+\alpha$ with $p \in \mathbb{Z}$ and $-1 < \alpha \leq 0$, the map

\vspace{0.2cm}
 
\begin{center}
$\partial_t^{n-p}:V^{\beta} \cap F^n\mathcal{H}_{X,0} \to V^{\alpha}/V^{>\alpha}=C^{\alpha}$
\end{center}

\vspace{0.2cm}

\noindent
induces an isomorphism from $Gr_{\beta}^V \Omega_f \to Gr_F^pC^{\alpha}$, and the diagram 

\vspace{0.3cm}

\begin{center}
$\begin{CD}
Gr_{\beta}^V \Omega_f @>Gr(f)>> Gr_{\beta+1}^V \Omega_f \\
@V\partial_t^{n-p}VV    @V\partial_t^{n-p+1}VV\\
Gr_F^pC^{\alpha} @>N_{p,a}>> Gr_F^{p-1}C^{\alpha}
\end{CD}$
\end{center}

\vspace{0.3cm}

\noindent
commutes up to a factor of $-2 \pi i$. Hence $Gr(f)$ and $Gr_FN$ have the same Jordan normal form.

\vspace{0.5cm}

\section{Integrals along Lefschetz thimbles}

\vspace{0.5cm}

The materials in this section are all well known, however we added this for the convenience in the notations used for the main theorem in section 5. Consider the function $f:\mathbb{C}^{n+1} \to \mathbb{C}$ with isolated singularity at $0$, and a holomorphic differential $(n+1)$-form $ \omega $ given in a neighborhood of the critical point. We shall study the asymptotic behaviour of the integral,
 
\begin{equation}
\displaystyle{\int_{\Gamma} e^{\tau f}\omega}
\end{equation}

\vspace{0.3cm}

\noindent
for large values of the parameter $\tau$, namely a complex oscillatory integral. Lets propose the same notation as section 1. In the long exact homology sequence of the pair $(X,X_{t})$ where $X$ is a tubular neighborhood of the singular fiber $X_0$ in the Milnor ball, 

\begin{equation}
... \to H_{n+1}(X) \to H_{n+1}(X,X_t) \stackrel{\partial_t}{\rightarrow} H_{n}(X_t) \to H_{n}(X) \to ...
\end{equation}

\vspace{0.3cm}

\noindent
$X$ is contractible. Therefore, we get an isomorphism $\partial_t:H_{n+1}(X,X_t) \cong H_{n}(X_t)$, and similar in cohomologies, i.e. $H^{n+1}(X,X_t) \cong H^{n}(X_t)$. If $\omega $ is a holomorphic differential $(n+1)$-form on $X$, and $\Gamma \in H_{n+1}(X,X_t)$, we have the following; 

\vspace{0.2cm}

\begin{proposition} (cf. \cite{AGV} Theorems 8.6, 8.7, 8.8, 11.2) Assume $\omega \in \Omega^{n+1}$, and let $\Gamma \in H_n(X,X_t)$. Then 

\begin{equation}
\displaystyle{\int_{\Gamma}e^{-\tau f}\omega}=\displaystyle{\int_{0}^{\infty}e^{-t \tau}\int_{\Gamma \cap \{f=t\}}\frac{\omega}{df}_{\mid_{X_t}}}dt=e^{\tau .f(0)} \displaystyle{\int_{\Gamma \cap \{f=t\}}\frac{\omega}{df}_{\mid_{X_t}}}
\end{equation}

\vspace{0.3cm}

\noindent
for $Re(\tau) $ large, and in this way can also be expressed as $\sum \tau^{\alpha} {\log \tau}^k A_{\alpha,k}$ in that range.
\end{proposition}

\vspace{0.5cm}

\noindent
By Proposition 2.1, we identify the cohomology classes $\displaystyle{\int_{\Gamma}e^{-\tau f}\omega}$ and $\displaystyle{\int_{\Gamma \cap \{f=t\}}\frac{\omega}{df}_{\mid_{X_t}}}$ via integration on the corresponding homology cycles. We can also choose $\Gamma $ such that its intersection with each Milnor fiber has compact support, and its image under $f$ is the positive real line, \cite{PH}. 

\vspace{0.5cm}

\noindent
The asymptotic integral 

\[ I(\tau)=\int e^{\tau f}\phi dx_0...dx_n , \qquad \tau \to \pm \infty \] 

\vspace{0.3cm}

\noindent
satisfies

\[ \dfrac{d^p}{d\tau^p}I=\int e^{ \tau f}f^p\phi dx_0...dx_n. \]

\vspace{0.3cm}

\noindent
In case $f$ is analytic then it has an asymptotic expansion 

\[ I(\tau)=\displaystyle{\sum_{\alpha,p,q}c_{\alpha,p,q}(f)\tau^{\alpha-p}(\log \tau)^q ,\qquad \tau \to +\infty} \]

\vspace{0.3cm}

\noindent
for finite number of rational numbers $\alpha <0, \ p \in \mathbb{N}, \ 0 \leq q \leq n-1$ . Then $\phi \to c_{\alpha,p,q}(\phi) $ is a distribution with support contained in the support of $f$, \cite{MA}.  

\vspace{0.5cm}

\begin{remark} (see \cite{PH} page 27)
We have the formula;

\vspace{0.3cm}

\begin{center}
$I(\tau )= (2 \pi)^{n/2}(Hess f)^{-1/2} f(0)\tau^{-n/2}[1+O(1/\tau)]$.
\end{center}

\vspace{0.3cm}

\end{remark}

\vspace{0.2cm}

\noindent
It follows that the form $e^{-\tau f}\omega$ (for $\tau $ large enough) and the form $\displaystyle{\frac{\omega}{df}}\mid_{X_t}$, define the same cohomology class via integration on cycles. 

\vspace{0.2cm}

\begin{proposition} (\cite{AGV} lemma 11.4, 12.2, and its corollary)
There exists a basis  $\ \omega_1,..., \omega_{\mu} \ $ of $\Omega_f$ such that the corresponding  Leray residues $\ \omega_1/df,..., \omega_{\mu}/df \ $ define a basis for the sections of vanishing cohomology.

\end{proposition}

\vspace{0.5cm}

\section{Extension of the Gauss-Manin system}

\vspace{0.5cm}

As mentioned in the introduction the Gauss-Manin system of an isolated hypersurface singularity can be extended over the puncture by a process of gluing vector bundles. The gluing is done by some comparison between $V$-lattices and the Brieskorn lattices in $\mathcal{G}$. We are interested to understand the fiber of $\mathcal{G}$ on $0$ after the extension.
Let $\Omega^{n+1}[\tau, \tau^{-1}]$ be the space of Laurent polynomials with coefficients in $\Omega^{n+1}$. According to its very definition ( \cite{PS} sec. 10.4, \cite{SAI6} lemma 2.4), the Gauss-Manin System is given by;

\vspace{0.3cm}

\begin{center}
$\mathcal{G}=\dfrac{\Omega^{n+1}[\tau,\tau^{-1}]}{(d-\tau df \wedge)\Omega^{n+1}[\tau,\tau^{-1}]}$, \\[0.5cm] 
$(d-\tau df \wedge)\sum_k\eta_k\tau^k=\sum_k(d\eta_k-df \wedge \eta_{k-1})\tau^k $.
\end{center}

\vspace{0.3cm}

\noindent
We put $\tau=t^{-1}$, and consider $(\tau, t)$ as coordinates on $\mathbb{P}^1$. Then $\mathcal{G}$ is a $\mathbb{C}[t,t^{-1}]$-module with connection and $\partial_{\tau}=-t^2\partial_t$. In the chart $t$, the Brieskorn lattice 

\vspace{0.3cm}

\begin{center}
$\mathcal{H}^{(0)}:=image (\Omega^{n+1}[\tau^{-1}] \to \mathcal{G})=\dfrac{\Omega^{n+1}[t]}{(td-df \wedge)\Omega^{n+1}[t]}$
\end{center}

\vspace{0.3cm}

\noindent
is a free $\mathbb{C}[t]$ module of rank $\mu$. It is stable by the action of $\partial_{\tau}=-t^2\partial_t$. Therefore $\partial_t$ is a connection on $\mathcal{G}$ with a pole of order $2$. We consider the increasing exhaustive filtration $\mathcal{G}_p:=\tau^p\mathcal{H}^{(0)}$ of $\mathcal{G}$.

In the chart $\tau$, there are various natural lattices indexed by $\mathbb{Q}$, we denote them by $V^{\alpha}$, with $V^{\alpha-1}=\tau V^{\alpha}$. On the quotient space $C_{\alpha}=V^{\alpha}/V^{>\alpha}$ there exists a nilpotent endomorphism $(\tau\partial_{\tau}-\alpha)$. The space $\oplus_{\alpha \in [0,1[}C_{\alpha}$ is isomorphic to $H^n(X_{\infty},{\mathbb{C}})$, and $\oplus_{\alpha \in [0,1[}F^pC_{\alpha}$  is the limit MHS on $H^n(X_{\infty},{\mathbb{C}})$. A basic isomorphism can be constructed, as 

\vspace{0.3cm}

\begin{center}
$\begin{CD}
\dfrac{\mathcal{G}_p\cap V^{\alpha}}{\mathcal{G}_{p-1}\cap V^{\alpha}+\mathcal{G}_p \cap V^{>\alpha}}=Gr_F^{n-p}(C_{\alpha}) \\   
@V\tau^pV{\cong}V \\
\dfrac{ V^{\alpha+p} \cap \mathcal{G}_0}{ V^{\alpha}\cap \mathcal{G}_{-1}+ V^{>\alpha} \cap \mathcal{G}_0}=Gr_{\alpha+p}^V(\mathcal{G}_0/\mathcal{G}_{-1})
\end{CD}$.
\end{center}

\vspace{0.3cm}

\noindent
Thus, the gluing is done via the isomorphisms, $ Gr_F^{n-p}(H_{\lambda}) \cong Gr_{\alpha+p}^V(\mathcal{H}^{(0)}/\tau^{-1}. \mathcal{H}^{(0)})$ where $\lambda=\exp{(2 \pi i \alpha)}$ and we have chosen $-1 \leq \alpha <0$ (cf. \cite{SA3}, \cite{SA5}, \cite{SA6}). We have
 
\begin{equation} 
\dfrac{\mathcal{H}^{(0)}}{\tau^{-1}.\mathcal{H}^{(0)}}= \dfrac{\Omega^{n+1}}{df \wedge \Omega^{n}} =\Omega_f 
\end{equation} 

\vspace{0.3cm}

\noindent
canonically. The identity (7) defines the extension fiber of the Gauss-Manin system of the isolated singularity $f:\mathbb{C}^{n+1} \to \mathbb{C}$. The same conclusion can be obtained when $f$ is a holomorphic germ, However one needs to consider the completions of the modules involved, (see \cite{MA} page 422 or \cite{S1}). In this way for $f:\mathbb{C}^{n+1} \to \mathbb{C}$ we have the same equation as (7). By identifying the sections with sections of relative cohomology via Theorem 3.1, this formula is a direct consequence of the formula 

\[ \displaystyle{\int_{\Gamma}e^{-\tau f}d\omega}=\tau \displaystyle{\int_{\Gamma}e^{-\tau f} df \wedge \omega} , \qquad \omega \in \Omega_X^n .\]

\vspace{0.3cm}

We refer to \cite{MA} page 422 for details.

\vspace{0.5cm}

\section{MHS on the extended fiber}

\vspace{0.5cm}

Asuume $f:\mathbb{C}^{n+1} \to \mathbb{C}$ is a germ of isolated singularity. Up to now we extended the local system of vanishing cohomology over the degeneracy point $0$. Now our task is how to set a mixed Hodge structure on that. We build an isomorphism 

\vspace{0.3cm}

\begin{center}
$\Phi:H^n(X_{\infty},\mathbb{C}) \to \Omega_f$
\end{center}

\vspace{0.3cm}

\noindent
which allows us to equip a mixed Hodge structure on $\Omega_f$, \cite{H1}, \cite{SAI6}. This also motivates the definition of opposite filtrations of M. Saito. It is based on the following theorem. 

\vspace{0.2cm}

\begin{proposition} (\cite{H1} prop. 5.1)
Assume $\{(\alpha_i,d_i)\}$ is the spectrum of a germ of isolated singularity $f:\mathbb{C}^{n+1} \to \mathbb{C}$. There exists elements $s_i \in C^{\alpha_i}$ with the properties

\vspace{0.2cm}

\begin{itemize}
\item[(1)] $s_1,...,s_{\mu}$ project onto a $\mathbb{C}$-basis of $\bigoplus_{-1<\alpha<n} Gr_V^{\alpha} H''/Gr_V^{\alpha} \partial_t^{-1}H''$. 

\item[(2)] $s_{\mu+1}:=0$; there exists a map $\nu:\{1,...,\mu\} \to \{1,...,\mu,\mu+1 \}$ with $(t-(\alpha_i+1)\partial_t^{-1})s_i =s_{\nu(i)}$

\item[(3)] There exists an involution $\kappa:\{1,...,\mu\} \to \{1,...,\mu \}$ with $\kappa(i)=\mu+1-i$ if $\alpha_i \ne \frac{1}{2}(n-1)$ and $\kappa(i)=\mu+1-i$ or $\kappa(i)=i$ if $\alpha_i = \frac{1}{2}(n-1)$, and  

\[ P_S(s_i,s_j)= \pm \delta_{(\mu+1-i)j}.\partial_t^{-1-n} \]

\vspace{0.2cm}

\noindent
where $P_S$ is the K. Saito higher residue pairing.
\end{itemize}

\end{proposition}

\vspace{0.2cm}

Suppose,

\vspace{0.3cm}

\begin{center}
$H^n(X_{\infty}, \mathbb{C})= \displaystyle{\bigoplus_{p,q,\lambda}} I^{p,q}_{\lambda}$ 
\end{center}

\vspace{0.3cm}

\noindent
is the Deligne-Hodge bigrading, and generalized eigen-spaces of vanishing cohomology, and also $\lambda=\exp(-2\pi i \alpha)$ with $\alpha \in (-1,0]$. Consider the isomorphism obtained by composing the three maps,

\vspace{0.3cm}

\begin{equation}
\Phi_{\lambda}^{p,q}: I^{p,q}_{\lambda} \stackrel{\hat{\Phi}_{\lambda}}{\longrightarrow}Gr_V^{\alpha+n-p}H'' \stackrel{pr}{\longrightarrow} Gr_V^{\bullet} H^{\prime \prime}/\partial_t^{-1}H^{\prime \prime} \stackrel{\cong}{\longrightarrow} \Omega_f
\end{equation}

\noindent
where 

\begin{center}
$\Phi=\bigoplus_{p,q,\lambda} \Phi_{\lambda}^{p,q}, \qquad \Phi_{\lambda}^{p,q}=pr \circ \hat{\Phi}_{\lambda}^{p,q}, \qquad \hat{\Phi}_{\lambda}^{p,q}:= \partial_t^{p-n} \circ \psi_{\alpha}|  I^{p,q}_{\lambda}$
\end{center}

\vspace{0.3cm}

\noindent
\textbf{Lemma:} The map $\Phi$ is a well-defined $\mathbb{C}$-linear isomorphism. 

\vspace{0.3cm}

\noindent
We list some of the properties of the map $\phi$ as follows;

\begin{itemize}

\vspace{0.3cm}

\item $\hat{\Phi}_{\lambda}^{p,q}$ takes values in $C^{\alpha+n-p}$. By the formula $F^p=\oplus_{r \geq p}I^{r,s}$, any cohomology class in $I_{\lambda}^{p,q}$, is of the form $\psi_{\alpha}^{-1} [\partial_t^{n-p} h^{\prime \prime}+V^{>\alpha}]=\psi_{\alpha}^{-1} \partial_t^{n-p} [h^{\prime \prime}+V^{>\alpha+n-p}]$, for $h^{\prime \prime} \in H^{\prime \prime}$, cf. def. 6.3.3. By substituting in the formula it explains the image of $\hat{\Phi}_{\lambda}^{p,q}$.
\item Taking two different representatives $\omega_1, \omega_2 \in \Omega_X^{n+1}$ for $ h''$ does not effect on the class $h''+V^{>\alpha+n-p}$. Because by identifying $H''$ with its image in $V^{-1}$, the difference $\omega_1- \omega_2$ belongs to $ V^{>\alpha+n-p}$.
\item The map $\Phi$ is obviously a $\mathbb{C}$-linear  isomorphism because both of the $\psi_{\alpha}$ and $\partial_t^{-1}$ are $\mathbb{C}$-linear isomorphisms on the appropriate domains, and 

\vspace{0.3cm}

\begin{center} 
$\Phi(I^{p,q}_{\lambda}) \subset \Phi(F^pH^n(X_{\infty})_{\lambda}) \subset V^{\alpha} \cap H^{\prime \prime}/V^{\alpha+1} \hookrightarrow  Gr_V^{\alpha}H^{\prime \prime}/\partial_t^{-1} H^{\prime \prime}$
\end{center}

\vspace{0.3cm}

\item The definition of $\Phi$ fixes an isomorphism $Gr_V^{\bullet} \Omega_f \cong \Omega_f$. On the eigen space $H_{\lambda}$ this corresponds to a choice of sections of 
$Gr_V^{\alpha}(V^{\alpha} \cap H'') \to Gr_V^{\alpha}H''/\partial_t^{-1}H''$ for $-1 \leq \alpha <0$.

\vspace{0.3cm}

\end{itemize} 

\vspace{0.5cm}

The mixed Hodge structure on $\Omega_f$ is defined by using the isomorphism $\Phi$. This means that
 
\vspace{0.3cm}
 
\begin{center}
$W_k(\Omega_f)=\Phi W_kH^n(X_{\infty},\mathbb{Q}),  
 \qquad F^p(\Omega_f)= \Phi F^pH^n(X_{\infty},\mathbb{C})$ 
\end{center}

\vspace{0.3cm}
 
\noindent 
and all the data of the Steenbrink MHS on $H^n(X_{\infty},\mathbb{C})$ such as the $\mathbb{Q}$ or $\mathbb{R}$-structure is transformed via the isomorphism $\Phi$ to that of $\Omega_f$. Specifically; in this way we also obtain a conjugation map

\vspace{0.3cm}

\begin{equation}
\bar{.}:\Omega_{f,\mathbb{Q}} \otimes \mathbb{C} \to \Omega_{f,\mathbb{Q}} \otimes \mathbb{C}, \qquad \Omega_{f,\mathbb{Q}}:=\Phi H^n(X_{\infty},\mathbb{Q})
\end{equation}

\vspace{0.3cm}

\noindent
defined from the conjugation on $H^n(X_{\infty},\mathbb{C})$ via this isomorphism. 

\vspace{0.2cm}

\noindent
The basis discussed in 4.1 is usually called a good basis. The condition (1) correspond to the notion of opposite filtrations. Two filtrations $F$ and $U$ on $H$ are called opposite (cf. \cite{SAI6} sec. 3) if 

\[ Gr_p^FGr_U^q H=0 , \qquad \text{for} \ p \ne q. \]

\vspace{0.3cm}

\noindent
When one of the filtrations is decreasing say $\{F^p\}$ and the other increasing say $\{U_q\}$ then this is equivalent to 

\vspace{0.3cm}

\begin{equation}
H=F^p\oplus U_{p-1}, \qquad \forall p.
\end{equation}

\vspace{0.3cm}

\noindent
Similarly, two decreasing filtrations $F$ and $U$ are said to be opposite if $F$ is opposite to the increasing filtration $U'_q:=U^{k-q}$, \cite{P2}. Note that $V^{\alpha} H''$ is the submodule generated by $s(V^{\alpha}\Omega_f)$. 

\vspace{0.2cm}

\begin{proposition} (\cite{SAI6} prop. 3.5)
The filtration

\vspace{0.3cm} 

\[ U^pC^{\alpha}:=C^{\alpha} \cap V^{\alpha+p}H'' \]

\vspace{0.2cm}

\noindent
is opposite to the filtration Hodge filtration $F$ on $\mathcal{G}$.
\end{proposition}

\vspace{0.3cm}

\noindent 
By the Proposition 4.2 the two filtrations $F^p$ and 

\vspace{0.2cm}

\begin{center}
$U'_q:=U^{n-q}=\psi^{-1} \{ \oplus_{\alpha} C^{\alpha} \cap V^{\alpha+n-q}H'' \} =\psi^{-1} \{\oplus_{\alpha}Gr_V^{\alpha}[V^{\alpha+n-q}H''] \}$ 
\end{center}

\vspace{0.2cm}

\noindent
are two opposite filtrations on $H^n(X_{\infty},\mathbb{C})$. We also have

\[ F^pH^n(X_{\infty},\mathbb{C})_{\lambda} \cong U'_pH^n(X_{\infty},\mathbb{C}). \]

\vspace{0.3cm}

\noindent
A standard example of such a situation is when the variation of MHS namely $\mathcal{H}$ is mixed Tate (also called Hodge-Tate). By definition a mixed Tate Hodge structure $H$ is when $Gr_{2l-1}^W H=0$ and $Gr_{2l}^WH =\oplus_i \mathbb{Q}(-n_i)$. In that case one easily shows the Deligne-Hodge decomposition becomes 

\[ \bigoplus_p (W_{2p}\cap F^p) H = H_{\mathbb{C}} \]

\noindent
and the two filtrations $F$ and $W$ are opposite.

In general 

\[ H'' \cap V^{\alpha}\mathcal{G} \ne (H'' \cap V^{\alpha}) + (H'' \cap V^{>\alpha}) \]

\vspace{0.3cm}

\noindent
This is why we have to take grading $Gr_V^{\alpha}$.

The complex structure defined on $\Omega_f$ via $\Phi:H^n(X_{\infty}) \cong \Omega_f$ is not unique, and it depends to the good basis chosen. However it does not affect the polarization, discussed in the next section.

\vspace{0.5cm}

\section{Polarization form on extension}

\vspace{0.5cm}

This section concerns the main contribution. Assume $f:\mathbb{C}^{n+1} \to \mathbb{C}$ is a germ of isolated singularity. In Section 4 we described the extension of MHS associated to $f$ over the degenerate point $0$. In Section 5 we defined a MHS on the new fiber appearing over $0$, namely $\Omega_f$. The MHS was defined by the specific isomorphism $\Phi$ in (9). We use the isomorphism 
$\Phi: H^n(X_{\infty}, \mathbb{C}) \to \Omega_f$  introduced in the previous section to express a correspondence between polarization form on vanishing cohomology and the Grothendieck pairing on $\Omega_f$.

\vspace{0.2cm}

\begin{theorem}  
Assume $f:(\mathbb{C}^{n+1},0) \to (\mathbb{C},0)$, is a holomorphic germ with isolated singularity at $0$. Then, the isomorphism $\Phi$ makes the following diagram commutative up to a complex constant;

\begin{equation}
\begin{CD}
\widehat{Res}_{f,0}:\Omega_f \times \Omega_f @>>> \mathbb{C}\\
@VV(\Phi^{-1},\Phi^{-1})V                   @VV \times *V \\
S:H^n(X_{\infty}) \times H^n(X_{\infty}) @>>> \mathbb{C}
\end{CD} \qquad \qquad  * \ne 0
\end{equation}

\vspace{0.3cm}

\noindent
where, 

\[ \widehat{Res}_{f,0}=\text{res}_{f,0}\ (\bullet,\tilde{C}\ \bullet) \]

\vspace{0.3cm}

\noindent
and $\tilde{C}$ is defined relative to the Deligne decomposition of $\Omega_f$, via the isomorphism $\Phi$. If $J^{p,q}=\Phi^{-1} I^{p,q}$ is the corresponding subspace of $\Omega_f$, then

\begin{equation}
\Omega_f=\displaystyle{\bigoplus_{p,q}}J^{p,q} \qquad \tilde{C}|_{J^{p,q}}=(-1)^{p} 
\end{equation}

\vspace{0.3cm}

\noindent
In other words;

\vspace{0.2cm}

\begin{equation}
S(\Phi^{-1}(\omega),\Phi^{-1}(\eta))= * \times \ \text{res}_{f,0}(\omega,\tilde{C}.\eta), \qquad 0 \ne * \in \mathbb{C}
\end{equation}

\vspace{0.3cm}

\end{theorem} 

\noindent
Part of this proof is given in \cite{CIR} for homogeneous fibrations in the context of mirror symmetry, see also \cite{PH}.

\begin{proof}

\noindent 
Before starting the proof lets mention that the map $\Phi$ is classically used to correspond the mixed Hodge structure on $H^n(X_{\infty},\mathbb{C})$ and $\Omega_f$. We only prove the correspondence on polarizations. 

\vspace{0.3cm}

\noindent
Step 1: Choose a $\mathbb{C}$-basis of the module $\Omega(f)$, namely $\{\phi_1,...,\phi_{\mu}\}$, where $\phi_i=f_i .d\underline{x}$. We identify the class $[e^{-f/t} \phi_i]$ with a cohomology class in $H(X_t)$. We may also choose the basis $\{ \phi_i \}$ so that the forms $\{\eta_i=e^{(-f/t)} \phi_i\}$ correspond to a basis of vanishing cohomology, by the formula

\begin{equation}
\displaystyle{\int_{\Gamma}e^{-\tau f}\omega}=\displaystyle{\int_{0}^{\infty}e^{-t \tau}\int_{\Gamma \cap X_t}\frac{\omega}{df}\mid_{X_t}}
\end{equation}

\vspace{0.3cm}

\noindent
Step 2: In this step, we assume the Poincar\`e product is non-degenerate. By this assumption we may also assume $f$ is homogeneous of degree $d$ and $\phi_i$'s are chosen by homogeneous basis of $\Omega_f$. Consider the deformation

\vspace{0.3cm}
 
\begin{center}
$f_s=f+\displaystyle{\sum_{i=0}^{n}}s_ix_i$ 
\end{center}

\vspace{0.3cm}

\noindent
and set 

\vspace{0.3cm}

\begin{center}
$S_{ij}(s,z):=\langle [e^{-f_s/z} \phi_i], [e^{+f_s/z} \phi_j] \rangle$. 
\end{center}

\vspace{0.3cm}

\noindent
The cup product is the one on the relative cohomology, and we may consider it in the projective space $\mathbb{P}^{n+1}$. The perturbation $f_s$ and also the Saito form $S_{ij}$ are weighted homogeneous. This can be seen by choosing new weights, $\deg(x_i)=1/d, \ \deg(s_i)=1-1/d$, and $\deg(z)=1$ then the invariance of the product with respect to the change of variable $x \to \lambda^{1/d}x, \ z \to \lambda z$, shows that $S_{ij}(s,z)$ is weighted homogeneous. We show that $S_{ij}$ is some multiple of $\widehat{Res}_{f,0}$.

\vspace{0.3cm}

\begin{center}
$S_{ij}(s,z):=(-1)^{n(n+1)/2}(2 \pi iz)^{n+1}(Res_f( \phi_i, \phi_j )+ O(z))$. 
\end{center}

\vspace{0.3cm}

\noindent
Suppose that $s$ is generic so that $ x \to Re( f_s/z) $ is a Morse function. Let $ \Gamma_1^+,...,\Gamma_{\mu}^+$, (resp. $\ \Gamma_1^-,...,\Gamma_{\mu}^-$) denote the Lefschetz thimbles emanating from the critical points $\sigma_1,...,\sigma_{\mu}$ of $Re(f_s/t)$ given by the upward gradient flow (resp. downward). Choose an orientation so that $\Gamma_r^+ . \Gamma_s^-= \delta_{rs}$. We have 

\vspace{0.3cm}

\begin{center}
$S_{ij}(s,z)= \displaystyle{\sum_{r=1}^{\mu}}(\displaystyle{\int_{\Gamma_r^+}}e^{-f_s/z}\phi_i)(\displaystyle{\int_{\Gamma_r^-}}e^{f_s/z}\phi_j)$-
\end{center}

\vspace{0.3cm}

\noindent
For a fixed argument of $z$ we have the stationary phase expansion as $z \to 0$.

\vspace{0.3cm}

\begin{center}
$(\displaystyle{\int_{\Gamma_r^+}}e^{-f_s/z}\phi_i) \cong \pm \frac{(2 \pi z)^{(n+1)/2}}{\sqrt{Hess f_s(\sigma_r)}}(f_i(\sigma_r) + O(z))$
\end{center}

\vspace{0.3cm}

\noindent
where $\phi_i=f_i(x) d{\underline{x}}$. Therefore,

\vspace{0.3cm}

\begin{center}
$S_{ij}(s,z)=(-1)^{n(n+1)/2}(2 \pi i z)^{n+1} \displaystyle{\sum_{r=1}^{\mu}}(\frac{f_i(\sigma_r)f_j(\sigma_r)}{Hess(f_s)(\sigma_r)}+O(z))$
\end{center}

\vspace{0.3cm}

\noindent
where the lowest order term in the right hand side equals the Grothendieck residue. As this holds for an arbitrary argument of $z$, and $S_{ij}$ is holomorphic for $z \in \mathbb{C}^*$; the conclusion follows for generic $s$. By analytic continuation the same holds for all $s$. By homogenity we get,

\begin{equation}
S_{ij}(0,z)=(-1)^{n(n+1)/2}(2 \pi iz)^{n+1}Res_f( \phi_i, \phi_j).
\end{equation}

\vspace{0.3cm}

\noindent
Note that there appears a sign according to the orientations chosen for the integrals; however this only modifies the constant in the theorem. Thus, we have;
\begin{equation}
S_{ij}(0,1)=(-1)^{n(n+1)/2}(2 \pi i)^{n+1}Res_f( \phi_i, \phi_j).
\end{equation}

\vspace{0.3cm}

\noindent
Step 3: The sign appearing in residue pairing is caused by compairing the 
two products 
\begin{equation}
(e^{-f} \phi_i,e^{-f} \phi_j), \qquad (e^{-f} \phi_i,e^{+f} \phi_j).
\end{equation}

\vspace{0.3cm}

\noindent
Assume we embed the fibration in a projective one as before, replacing $f$ with a homogeneous polynomial germ of degree $d$. We can consider a change of variable as $I:z \to e^{\pi i/d}z$ which changes $f$ by $-f$. Thus this map is an involution on the value of $f$. The cohomology bases in $Gr_F^pGr_{n+1}^W H^n \subset I^{p,q}$ can be charcterized by the degree of forms, i.e. by the  $\phi_i=f_i dz$, with $f_i$ homogeneous, and $ p \leq deg(\phi_i)=l(\phi_i) < p+1$. It follows that the cohomology class $e^{-f} \phi_j$ after the this change of variable is replaced by $c_p. e^{+f} \phi_j$ where $c_p \in \mathbb{C}$ only depends to the Hodge filtration (defined by degree of forms). By the change of variable we obtain;

\begin{equation}
(e^{-f} \phi_i,I^*e^{-f} \phi_j)=(e^{-f} \phi_i, (-1)^{\deg \phi_j/d} e^{+f} \phi_j) 
\end{equation}

\vspace{0.3cm}

\noindent
because $I^d=id$, if we iterate $I^*$, $d$ times we obtain;

\[ (e^{-f} \phi_i,e^{-f} \phi_j)= res_{f,0}(a\ ,(-1)^{(d-1)\deg \phi_j/d}.b) 
 \]

\vspace{0.3cm}
 
\noindent 
The Riemann-Hodge bilinear relations in $H_{\ne 1}$ implies that, the products of the forms under consideration is non-zero except when the degrees of $\phi_i$ and $\phi_j$ sum to $n$. This explains the formula in $H_{\ne 1}$. The above argument will still hold when the form is replaced by $(\bullet, N_Y \bullet)$, by the linearity of $N_Y$. Thus, we still have the same result on $H_{\ne 1}$.

\vspace{0.3cm}

\noindent
Step 4: In case the Poincar\`e product is degenerate, we still assume $f$ is homogeneous but we change the cup product by applying $N_Y$ on one component. The same relation can be proved between the level form $(\bullet, N_Y \bullet)$, and the corresponded local residue, i.e.

\[ S_Y(\Phi^{-1}(\omega),\Phi^{-1}(\eta)) =*.\ \widehat{Res}(a\ ,\tilde{C}.b) , \qquad *\ne 0, \ a,b \in \Omega_f \]

\vspace{0.3cm}

\noindent
implies   

\[ S_Y(\Phi^{-1}(\omega),N_Y. \Phi^{-1}(\eta)) =*.\ \widehat{Res}(a\ ,\mathfrak{f}.\tilde{C}.b) , \qquad * \ne 0 , \ \ a,b \in \Omega_f \]

\vspace{0.3cm}

\noindent
where $\mathfrak{f}$ is the nilpotent transformation corresponded to $N_Y$ via $\Phi$.

\vspace{0.3cm}

\noindent
Step 5: By a standard theorem of J. Scherk, \cite{V} page 38, \cite{SC2} , (see also the corollary in \cite{V} sec. 3 page 37), continuity of Grothendieck residue, \cite{G3} page 657, after embedding of the Milnor fibration of $f$ into that of $f_Y$, the Grothendieck pairing for $f_Y$ is the prolongation of that of $f$. 

\vspace{0.3cm}

\noindent
Step 6: Until now we have proved the relation 

\begin{equation}
S(\Phi^{-1}(\omega),\Phi^{-1}(\eta))= * \times \ \text{res}_{f_Y,0}(\omega,(-1)^{p(d-1)/d}.\eta), \qquad 0 \ne * \in \mathbb{C}
\end{equation}

\noindent
For some $d$ and moreover, $d$ can be as large as we like. Because the left hand side is independent of $d$, if we let $d \to \infty$
then by Step 5 we obtain (13).

\end{proof}

\begin{example}
According to the above description the isomorphism $\Phi$ is as follows,

\[
\Phi^{-1}: [z^i dz] \longmapsto c_i. [res_{f=1}( z^i dz/(f-1)^{[l(i)]}) ]
\]

\vspace{0.3cm}

\noindent
with $c_i \in \mathbb{C}$, and $z^i$ in the basis mentioned above (see \cite{CIR}, Appendix A). For instance by taking $f=x^3+y^4$, then as basis for Jacobi ring, we choose 

\vspace{0.3cm}

\begin{center}
$z^i:\ 1, \ y, \ x, \ y^2, \ xy, \ xy^2$ 
\end{center}

\vspace{0.3cm}

\noindent
which correspond to top forms with degrees 

\vspace{0.3cm}

\begin{center}
$l(i): \ 7/12, \ 10/12, \ 11/12, \ 13/12, \ 14/12,\ 17/12$
\end{center} 

\vspace{0.3cm}

\noindent
respectively. The above basis projects onto a basis

\vspace{0.3cm}
 
\begin{center}
$\displaystyle{\bigoplus_{-1 < \alpha=l(i) -1 < n}Gr_{\alpha}^VH'' \twoheadrightarrow Gr_V \Omega_f}$ 
\end{center}

\vspace{0.3cm}

\noindent
as in Theorem 5.1. The Hodge filtration is explained as follows. First we have $h^{1,0}=h^{0,1}=3$. Therefore, because $\Phi$ is an isomorphism. 

\[ <1 .\omega,\ y. \omega, \ x .\omega > =\Omega_f^{0,1}, \qquad <y^2. \omega, \ xy.\omega  , \ xy^2.\omega > =\Omega_f^{1,0} \]

\vspace{0.3cm}

\noindent
where $\omega =dx \wedge dy$, and the Hodge structure is pure, because $Gr_2^WH^n(X_{\infty})=0$.

\vspace{0.3cm}

\begin{center}
$\overline{<1. dx \wedge dy, \ y. dx \wedge dy, \ x. dx \wedge dy>}= $ \\ $<c_1. xy^2 .dx \wedge dy, \ xy .dx \wedge dy , \ y^2 .dx \wedge dy> $ 
\end{center}

\vspace{0.3cm}

\end{example}

\vspace{0.2cm}

\begin{theorem} 

\noindent
Assume $f:\mathbb{C}^{n+1} \to \mathbb{C}$ is a holomorphic isolated singularity germ. The modified Grothendieck residue provides a polarization for the extended fiber $\Omega_f$, via the aforementioned isomorphism $\Phi$. Moreover, there exists a unique set of forms $\{\widehat{Res}_k\}$ polarizing the primitive subspaces of $Gr_k^W\Omega_f$ providing a graded polarization for $\Omega_f$.

\end{theorem}

\begin{proof}
Because $H^n(X_{\infty})$ is graded polarized, hence using theorem 5.1 $\Omega_f$ is also graded polarized via the isomorphism $\Phi$. By the Mixed Hodge Metric theorem, the Deligne-Hodge decomposition; 

\begin{equation}
\Omega_f=\displaystyle{\bigoplus_{p,q}}J^{p,q} 
\end{equation}

\vspace{0.3cm}

\noindent
is graded polarized and there exists a unique Hermitian form; $\mathcal{R}$ with,

\begin{equation}
i^{p-q}\mathcal{R}(v,\bar{v})>0, \qquad v \in J^{p,q}
\end{equation}

\vspace{0.3cm}

\noindent
and the decomposition is orthogonal with respect to $\mathcal{R}$. Here the conjugation is that in (9). This shows that the polarization forms $\{\widehat{Res}_l\}$ are unique. 

\noindent
Let $N:=\log M_u$ be the logarithm of the unipotent part of the monodromy for the Milnor fibration defined by $f$.  We have  

\[ H^n(X_{\infty})=\bigoplus_r N^rP_{l-2r}, \qquad P_l:=\ker N^{l+1}:Gr_l^W H^n \to Gr_{-l-2}^WH^n \]

\vspace{0.3cm}

\noindent
and the level forms 

\[ S_l:P_l \otimes P_l \to \mathbb{C}, \qquad S_l(u,v):=S(u,N^lv) \]

\vspace{0.3cm}

\noindent
polarize the primitive subspaces $P_l$. By using the isomorphism $\Phi$, similar type of decomposition exists for $\Omega_f$. That is the isomorphic image $P_l^{\prime}:=\Phi^{-1}P_l$ satisfies

\[ Gr_l^W\Omega_f=\bigoplus_r N^rP_{l-2r}^{\prime}, \qquad P_l^{\prime}:=\ker \mathfrak{f}^{l+1}:Gr_l^W \Omega_f \to Gr_{-l-2}^W \Omega_f \]

\vspace{0.3cm}

\noindent
and the level forms 

\[ \widehat{Res}_l:P_l^{\prime} \otimes P_l^{\prime} \to \mathbb{C}, \qquad \widehat{Res}_l:=\widehat{Res}(u,\mathfrak{f}^lv) \]

\vspace{0.3cm}

\noindent
polarize the primitive subspaces $P_l^{\prime}$, where $\mathfrak{f}$ is the map induced from multiplication by $f$ on $Gr_l^W \Omega_f$. Specifically, this shows

\vspace{0.3cm}

\begin{itemize}
\item  $\widehat{Res}_l(x,y)=0, \qquad x \in P_r^{\prime}, y \in P_s^{\prime}, r \ne s$
\item  $ Const \times \widehat{Res}_l(C_lx,\mathfrak{f}^l\bar{x}) >0, \qquad 0 \ne x \in P_l^{\prime}$

\vspace{0.2cm}

\noindent
where $C_l$ is the corresponding Weil operator.
\end{itemize}

\vspace{0.3cm}

\end{proof}

\vspace{0.2cm}

\noindent
Using this corollary and summing up all the material in sections 4 and 5 , we can give the following picture for the extension of PVMHS associated to isolated hypersurface singularity.

\vspace{0.2cm}

\begin{theorem}
Assume $f:\mathbb{C}^{n+1} \to \mathbb{C}$ is a holomorphic hypersurface germ with isolated singularity at $0 \in \mathbb{C}^{n+1}$. Then the variation of mixed Hodge structure, cf. \cite{H1}. This VMHS can be extended to the puncture with the extended fiber isomorphic to $\Omega_f$ in the sense of sec. 4 and 5, and it is polarized by 5.1 and 5.4. The Hodge filtration on the new fiber $\Omega_f$ corresponds to an opposite Hodge filtration on $H^n(X_{\infty},\mathbb{C})$ in the way explained in 4.2.
\end{theorem}

\vspace{0.2cm}

The Riemann-Hodge bilinear relations for the MHS on $\Omega_f$ and its polarization $\widehat{Res}$ can be deduced from that of $(H^n(X_{\infty}),S)$. 

\vspace{0.2cm}

\begin{corollary} (Riemann-Hodge bilinear relations for modified Grothendieck)
Assume $f:\mathbb{C}^{n+1} \to \mathbb{C}$ is a holomorphic germ with isolated singularity. Suppose $\mathfrak{f}$ is the corresponding map to $N$ on $H^n(X_{\infty})$, via the isomorphism $\Phi$. Define 

\[ P_l=PGr_l^W:=\ker(\mathfrak{f}^{l+1}:Gr_l^W\Omega_f \to Gr_{-l-2}^W\Omega_f) \]

\vspace{0.3cm}

\noindent
Going to $W$-graded pieces;
\begin{equation}
\widehat{Res}_l: PGr_l^W \Omega_f \otimes_{\mathbb{C}} PGr_l^W \Omega_f \to \mathbb{C}
\end{equation}

\vspace{0.3cm}

\noindent
is non-degenerate and according to Lefschetz decomposition 

\[ Gr_l^W\Omega_f=\bigoplus_r \mathfrak{f}^r P_{l-2r} \]

\vspace{0.3cm}

\noindent
we will obtain a set of non-degenerate bilinear forms,

\begin{equation}
\widehat{Res}_l \circ (id \otimes \mathfrak{f}^l): P Gr_l^W \Omega_f  \otimes_{\mathbb{C}} P Gr_l^W \Omega_f  \to \mathbb{C}, 
\end{equation} 

\begin{equation}
\widehat{Res}_l=res_{f,0}\ (id \otimes \tilde{C} .\  \mathfrak{f}^l)
\end{equation}

\vspace{0.3cm}

\noindent
where $\tilde{C}$ is as in 5.1, such that the corresponding Hermitian form associated to these bilinear forms is positive definite. In other words, 

\vspace{0.3cm}

\begin{itemize}

\item $\widehat{Res}_l(x,y)=0, \qquad x \in P_r, \ y  \in P_s, \ r \ne s $

\item If $x \ne 0$ in $P_l$, 

\[ Const \times res_{f,0}\ (C_lx,\tilde{C} .\  \mathfrak{f}^l .\bar{x})>0  \]

\vspace{0.3cm}

\noindent
where $C_l$ is the corresponding Weil operator. 

\end{itemize}

\vspace{0.3cm}

\end{corollary}

\begin{proof}
This follows directly from 5.1 and 5.4.
\end{proof}

\vspace{0.2cm}

\begin{example}
According to Example 5.2, if $f=x^3+y^4$ for instance
$\widehat{Res}(x.\omega,xy.\omega) =0$, and

\begin{center} 
$* \times \widehat{Res}(1.\omega,\overline{1.\omega}) >0$\\[0.2cm]
$* \times \widehat{Res}(x.\omega,\overline{x.\omega}) >0$\\[0.2cm]
$* \times \widehat{Res}(y.\omega,\overline{y.\omega}) >0$
\end{center}

\noindent
simultaneously for one $* \in \mathbb{C}$. 

\end{example}

\vspace{0.2cm}

The conjugation on the elementary sections of the Deligne  extension is done by the following formulas

\vspace{0.3cm}

\begin{center} 
$\psi_{\alpha}^{-1} (\overline{t^{\alpha}(\log t)^l A_{\alpha,l}})=\begin{cases} \psi_{1-\alpha}^{-1}( t^{1-\alpha}(\log t)^{\nu-l} \bar{A}_{1-\alpha,\nu-l}) , \qquad \alpha \in ]0,1[ \\
\psi_{0}^{-1}((\log t)^l \bar{A}_{0,\nu-l}), \ \  \qquad \qquad \qquad \alpha=0 \end{cases}$
\end{center}

\vspace{0.3cm}

\noindent 
where $\nu $ is the size of the corresponding Jordan block. Regarding the map $\Phi$ defined in Sec. 4, the conjugation on $\Omega_f$ must should satisfy similar relations. That is the conjugate of an element in $Gr_V^{\alpha}Gr_l^W\Omega_f$ is either in $Gr_V^{1-\alpha}Gr_{\nu-l}^W\Omega_f$ or $Gr_V^{0}Gr_{\nu-l}^W\Omega_f$, in respective cases, such that the corresponding sections of vanishing cohomology satisfy the above.

\vspace{0.2cm}

Consider the map 

\[ F:\Omega_X^{n+1} \to i_* \bigcup_z Hom \big (H_n(X,f^{-1}(\eta.\dfrac{z}{|z|}),\mathbb{Z}) \cong \oplus_i \mathbb{Z}\Gamma_i, \mathbb{C} \big ), \qquad \mathcal{H}:=Im(F)\] 

\[ \omega \mapsto [z \to (\Gamma_i \to \int_{{\Gamma}_i}e^{-t/z}\omega)],  \]

\vspace{0.3cm}

\noindent
where $\Gamma_i$ are the classes of Lefschetz thimbles. The vector bundle $\mathcal{H}$ is exactly the Fourier-Laplace transform of the cohomology bundle $R^nf_*\mathbb{C}_{S^*}=\cup_tH^n(X_t,\mathbb{C})$, equipped with a connection with poles of order at most two at $\infty$.

\[ (\cup_tH^n(X_t,\mathbb{C}), \nabla) \leftrightarrows (\mathcal{H}, \nabla^{\prime}) \]

\vspace{0.3cm}

\noindent
The modified Grothendieck residue

\[ \widehat{Res}_{f,0}=res_{f,0}(\bullet,\tilde{C}\bullet) \]

\vspace{0.3cm}

\noindent
with $\tilde{C}$ as in 6.1, is the Fourier-Laplace transform of the polarization $S$ on $H^n(X_{\infty},\mathbb{C})$ (\cite{DW} page 53, 54, prop. 2.6 - \cite{SA7} sec. 5).

\vspace{0.5cm}

\section{Asymptotic Hodge theory} 

\vspace{0.5cm}

A variation of graded polarized $\mathbb{C}$-mixed Hodge structure  is defined analogously having horizontality for $F$, and a collection of $(Gr_k^{\mathcal{W}},\mathcal{F}Gr_k^{\mathcal{W}},Q_k)$ of pure polarized $\mathbb{C}$-Hodge structures. One of the important facts in asymptotic Hodge theory and also Mirror symmetry is the following.

\vspace{0.2cm}

\begin{theorem} (P. Deligne)
Let $\mathcal{V} \to \triangle^{*n}$ be a variation of pure polarized Hodge structure of weight $k$, for which the associated limiting mixed Hodge structure is Hodge-Tate. Then the Hodge filtration $\mathcal{F}$ pairs with the shifted monodromy weight filtration $\mathcal{W}[-k]$, of $\mathcal{V}$, to define a Hodge-Tate variation over a neighborhood of $0$ in $\triangle^{*n}$. 
\end{theorem}

\vspace{0.2cm}

\begin{theorem} (\cite{P2} Theorem 3.28)
Let $\mathcal{V}$ be a variation of mixed Hodge structure, and 
\begin{center}
$\mathcal{V}=\displaystyle{\bigoplus_{p,q}}I^{p,q}$
\end{center}
denotes the $C^{\infty}$-decomposition of $\mathcal{V}$ to the sum of $C^{\infty}$-subbundles, defined by point-wise application of Deligne theorem. Then the Hodge filtration $\mathcal{F}$ of $\mathcal{V}$ pairs with the increasing filtration

\begin{equation}
\bar{U}_q=\sum_k \bar{\mathcal{F}}^{k-q} \cap \mathcal{W}_k
\end{equation}

\vspace{0.3cm}

\noindent
to define an un-polarized $\mathbb{C}VHS$.
\end{theorem}

\vspace{0.2cm}

Given a pair of increasing filtrations $A$ and $B$ of a vector space $V$ one can define the convolution $A*B$ to be the increasing filtration $A*B=\sum_{r+s=q}A_r \cap B_s.$

\vspace{0.2cm}

\begin{theorem} (G. Pearlstein-J. Fernandez)\cite{P2}
Let $\mathcal{V}$ be an admissible variation of graded polarized mixed Hodge structures with quasi-unipotent monodromy, and $\mathcal{V}=\oplus I^{p,q}$ the decomposition relative to the limiting mixed Hodge structure. Define

\begin{equation}
U'_p=\displaystyle{\bigoplus_{a \leq p}}I^{p,q}
\end{equation}

\vspace{0.3cm}

\noindent
and $\mathfrak{g}_{-}=\{\alpha \in \mathfrak{g}_{\mathbb{C}} | \alpha(U'_p) \subset U'_{p-1}\}$, then;

\vspace{0.3cm}

\begin{itemize}

\item[(a)] $U'$ is opposite to $F_{\infty}$. Moreover, relative to the decomposition
\begin{equation}
\mathfrak{g}=\displaystyle{\bigoplus_{r,s}}\mathfrak{g}^{r,s}.
\end{equation}

\vspace{0.3cm}

\item[(b)] If $\psi(s):\triangle^{*n} \to \check{D}$ is the associated untwisted period map, then in a neighborhood of the origin it admits a unique representation of the form 
\begin{equation}
\psi(s)=e^{\Gamma(s)}.F_{\infty}
\end{equation}

\vspace{0.3cm}

\noindent
where $\Gamma(s)$ is a $\mathfrak{g}_{-}$-valued function.

\item[(c)] $U'$ is independent of the coordinate chosen for $F_{\infty}$. Moreover,

\begin{equation}
U'=\overline{F_{nilp}^{\vee}}*W=\overline{F_{\infty}^{\vee}}*W.
\end{equation}

\vspace{0.3cm}

\end{itemize}
\end{theorem}

\vspace{0.2cm}

\noindent
Here above $F_{nilp}$ is an arbitrary element in the nilpotent orbit of the limit Hodge filtration corresponded to the nilpotent cone (i.e. positive linear combination) of the logarithms of the generators of the monodromy group, i.e $F_{nilp}=exp(z_1N_1+...+z_rN_r)F_{\infty}$ where $N_k$ are logarithms of different local monodromies, cf. \cite{P2}. 

\vspace{0.2cm}

\begin{theorem} Let $\mathcal{V}$ be an admissible variation of polarized mixed Hodge structure associated to a holomorphic germ of an isolated hyper-surface singularity. Set 

\begin{equation}
U'=\overline{F_{\infty}^\vee}*W.
\end{equation}

\vspace{0.3cm}

\noindent
Then $U'$ extends to a filtration $\underline{U'}$ of $\mathcal{V}$ by flat sub-bundles, which pairs with the limit Hodge filtration $\mathcal{F}$ of $\mathcal{V}$, to define a polarized $\mathbb{C}$-variation of Hodge structure, on a neighborhood of the origin.
\end{theorem}

The polarization of a complex variation of Hodge structure will probably be interpreted to mean a parallel hermitian form which makes the system of Hodge bundles $\mathcal{V}^p$ orthogonal, and becomes positive definite on multiplying the form by $(-1)^p$ on $\mathcal{V}^p$. Suppose that in the situation of Theorem 8.6 there is any such hermitian form $\mathcal{R}$. Then, on the one hand since $\mathcal{R}$ and $U'$ are flat, so is the orthogonal complement of $U'_{p-1}$ in $U'_p$. On the other hand, the way things have been setup, the orthogonal complement of $U'_{p-1}$ in $U'_p$ is exactly

\begin{equation}
\mathcal{V}^p = U'_p \cap F^p.
\end{equation}

\vspace{0.3cm}

\noindent
But this is the system of Hodge bundles, and so the Hodge filtration is also flat. The above discussion also proves the following,

\vspace{0.2cm}

\begin{corollary}
The mixed Hodge structure on the extended fiber $\Omega_f$
defined in Sec. 4, can be identified with 

\[ \Phi(U'=\overline{F_{\infty}^\vee}*W) \]

\vspace{0.3cm}

\noindent
where $\Phi$ is as in 8.5.
\end{corollary}

\begin{proof}
The corollary is a consequence of 6.3, 6.4 (c) and 6.5.
\end{proof}

\vspace{0.2cm}

By a theorem of Kattani-Caplan-Schmid, \cite{CKS}, there exists a unique nilpotent transformation $\delta $ in the Lie algebra of the linear endomorphisms $End_{\mathbb{R}}(\Omega_f)^{-1,-1}$ preserving the weight filtration, s.t. $(W,e^{-2i \delta}.F)$ is a mixed Hodge structure splits over $\mathbb{R}$ on $\Omega_f$. Then,

\[ \tilde{F}:=e^{i.\delta}.F. \]

\vspace{0.3cm}

\noindent
Since $\delta \in \mathfrak{g}_{\mathbb{R}}^{-1,-1} \subset W_{-2}^{\mathfrak{gl}}$, this element leaves $W$ invariant and acts trivially on the quotient $Gr_l^W$. Therefore both $F,\tilde{F}$ induce the same filtrations on $Gr_l^W H$. The bigrading $J_1^{p,q}$ defined by $J_1^{p,q}:=e^{-i.\delta}.J^{p,q}$ is split over $\mathbb{R}$. The operator $\tilde{C}_1=Ad(e^{-i.\delta}).\tilde{C}:\Omega_f \to \Omega_f$ defines a real structure on $\Omega_f$. If $\Omega_{f,1}=\oplus_{p < q} J_1^{p,q}$ then 

\[ \Omega_f=\Omega_{f,1} \oplus \overline{\Omega_{f,1}} \oplus \bigoplus_pJ_1^{p,p}, \qquad \overline{J_1^{p,p}}=J_1^{p,p}. \]

\vspace{0.3cm}

\noindent
The statement of theorem 5.1 is valid when the operator $\tilde{C}$ is replaced with $\tilde{C}_1$; 

\begin{equation}
S(\Phi^{-1}(\omega),\Phi^{-1}(\eta))= * \times \ \text{res}_{f,0}(\omega,\tilde{C}_1.\eta), \qquad 0 \ne * \in \mathbb{C}
\end{equation}

\vspace{0.3cm}

\noindent
and this equality is defined over $\mathbb{R}$. The content of the Theorem 8.1 is related to the $\mathfrak{sl}_2$-orbit theorem. The real splitting $I_1^{q,p} =\overline{I_1^{p,q}}$ corresponds to a semisimple transformation $Y_1 .v=(p+q-k) .v$ for $v \in I_1^{p,q}$. Then the pair $\{Y_1,N\}$ can be completed to an $\mathfrak{sl}_2$-triple $\{N_1^{+}, Y_1,N\}$. $N_1^{+} $ is real and $\delta,Y_1, N_1^+ \in \mathfrak{g}_{\mathbb{R}}$ are infinitesimal isometries of the polarization \cite{CKS} page 477. This shows that $\Omega_f$ can be equipped with a MHS that is real split and is a sum of pure Hodge structures.

\[ \Omega_f=\bigoplus_k \bigoplus_{p+q=k} J_1^{p,q} \]

\vspace{0.3cm}

\noindent
and also an $\mathfrak{sl}_2$-triple $\{ \mathfrak{f}_1^+,Y_1', \mathfrak{f} \}$ as infinitesimal isometries of the bilinear form $\widehat{Res}_{f,0}$ which are morphisms of the MHS $(F_1',W)$ explained above and are of types $(1,1),(0,0), \\ (-1,-1)$ respectively.

\vspace{0.5cm} 

\section{Applications}

We provide two applications of the procedure explained in section 4, 5, and 6.

\vspace{0.3cm}

(1) Let $C$ be a smooth projective curve over the field $\mathbb{C}$, and $J(C)$ its Jacobian. Then it is well known fact that the Poincar\`e duality of $H^1(C,\mathbb{C})$ can be identified with the polarization of $J(C)$, given by the $\Theta$-divisor. We want to consider this situation in a family $C_s$ parametrized by a 1-dimensional variety $S$. Suppose that 

\[ J^1(H_s^1)=H^1_{s,\mathbb{Z}} \setminus H_{s,\mathbb{C}}^1/F^0H_{s,\mathbb{C}}^1 \]

\[ J(\mathcal{H})= \displaystyle{\bigcup_{s \in S^*}J^1(H_s)} \]

\vspace{0.3cm}

\noindent
is the family of Jacobians associated to the variation of Hodge structure in a projective degenerate family of algebraic curves (here we have assumed the Hodge structures have weight -1), and $\dim(S)=1$. Then the fibers of this model are principally polarized abelian varieties. The polarization of each fiber is given by the Poincar\`e product of the middle cohomology of the curves, via a holomorphic family of $\Theta$-divisors. We are going to apply the construction in sec. 4 and 5 to the variation of Jacobians. 
\vspace{0.2cm}

\begin{theorem}
Consider a projective family of smooth curves parametrized by the manifold $S$ where $\dim(S)=1$ and degenerate over an isolated point in $S$. Let the Jacobian bundle $J(\mathcal{H})$ be defined as above. Then $J(\mathcal{H})$ extends over the degeneracy point as the Jacobian of the Hodge structure on $\Omega_f$, where $f$ is the local holomorphic map defining the fibration.
\end{theorem}

\begin{proof}
To extend $J(\mathcal{H})$ to a space over $S$, we let $\mathcal{G}$ be the Gauss-Manin system on $S^*$, obtained from the variation $\mathcal{H}$. 
On $S^*$ we have an extension of integral local classes 

\begin{equation}
0 \to \mathcal{H}_s \to \mathcal{J}_s \to \mathbb{Z}_s \to 0. 
\end{equation}

\vspace{0.2cm}

\noindent
On the Gauss-Manin systems we get

\begin{equation} 
0 \to M \to N \to \mathbb{Q}_{S^*}^H[n] \to 0 
\end{equation}

\vspace{0.2cm}

\noindent
with $\mathbb{Q}_S^H[n]$ is the trivial module with sheaf of sections $\mathcal{O}_{S^*}$. The first and the last objects in the short exact sequences (33) and (34) extend to the punctures in a way that the extended fiber is polarized by modified Grothendieck pairing. The extended fiber of the Jacobian bundle is the Jacobian of the opposite Hodge filtration. In this way the extended fiber is an abelian variety and principally polarized, with some $\Theta$-divisor. The extended Jacobian simply is 

\[ X_0=J^1(\Omega_f)=\Omega_{f,\mathbb{Z}} \setminus \Omega_f/F^0\Omega_f. \]

\vspace{0.5cm}

At the level of local systems (Hodge structures) we have 
the diagram of flat pairings,

\vspace{0.1cm}
 
\begin{equation}
\begin{array}[c]{cccccc}
\kappa: & \mathcal{H} & \otimes & \mathcal{H} & \rightarrow & \mathbb{C}\\
& \downarrow &  & \downarrow &  & \\
\kappa_J: & J & \otimes & J & \rightarrow & \mathbb{C}\\
& \downarrow &  & \downarrow &  & \\
\times : & \mathbb{Q} & \otimes & \mathbb{Q} & \rightarrow & \mathbb{C}\\ 
\end{array}.
\end{equation}

\vspace{0.3cm}

\noindent
The extension of the first and the last provides an extension of the middle line. Similar non-degenerate bilinear forms can be defined on the Gauss-Manin modules, where the above diagram is its reduction on fibers;

\vspace{0.1cm}

\begin{equation}
\begin{array}[c]{cccccc}
K: & \mathcal{G} & \otimes & \mathcal{G} & \rightarrow & \mathbb{C}[t,t^{-1}]\\
& \downarrow &  & \downarrow &  & \\
K_J: & N & \otimes & N & \rightarrow & \mathbb{C}[t,t^{-1}]\\
& \downarrow &  & \downarrow &  & \\
\times : & \mathbb{Q}_{S}^H & \otimes & \mathbb{Q}_{S}^H & \rightarrow & \mathbb{C}[t,t^{-1}]
\end{array}
\end{equation}

\vspace{0.1cm}

\noindent
where the map in the first line is the K. Saito higher residue pairing. 

\end{proof}

\vspace{0.5cm}

(2) Hodge theory assigns to any polarized Hodge structure $(H,F,S)$ a signature which is the signature of the hermitian form $S(C \bullet, \bar{\bullet})$, where $C$ is the Weil operator. In case of a polarized mixed Hodge structure $(H,F,W(N),S)$, where $N$ is a nilpotent operator this signature is defined to be the sum of the signatures of the hermitian forms associated to the graded polarizations $S_l: PGr_l^W H \times PGr_l^W H \to \mathbb{C}$, i.e signatures of $h_l:=S_l(C_l \bullet , N^l \bar{\bullet})$ for all $l$. A basic example of this is the signature associated to mixed Hodge structure on the total cohomology of a compact Kahler manifold, namely Hodge index theorem. In this case the MHS is polarized by 

\[ S(u,v)=\displaystyle{(-1)^{m(m-1)/2}\int_{X}} u \wedge v \wedge \omega^{n-m}, \qquad u,v \in H^m. \]

\vspace{0.3cm}

\noindent
where $\omega$ is the Kahler class. The signature associated to the polarization $S$ is calculated by W. Hodge in this case.

\vspace{0.5cm}

\begin{theorem} (W. Hodge)
The signature associated to the polarized mixed Hodge structure of an even dimensional compact Kahler manifold is $\sum_{p,q}(-1)^qh^{p,q}$, where the sum runs over all the Hodge numbers, $h^{p,q}$. This signature is $0$ when the dimension is odd.
\end{theorem}

\vspace{0.5cm}

\noindent
Similar definitions can be applied to polarized variation of mixed Hodge structure, according to the invariance of Hodge numbers in a variation of MHS. In the special case  of isolated hypersurface sigularities, the polarization form is given by $S_{ \ne 1} \oplus S_1$ where

\vspace{0.3cm}

\begin{center}
$ S_{\ne 1}(a,b) = S_Y(i^*a,i^*b), \qquad a,b \in H^n(X_{\infty})_{\ne 1} $ \\[0.2cm]
$S_{1}(a,b) = S_Y(i^*a,i^*N_Yb), \qquad a,b \in H^n(X_{\infty})_{1} $.
\end{center}

\vspace{0.3cm}

\noindent
A repeated application of theorem 7.2 gives the following,   

\vspace{0.2cm}

\begin{theorem} \cite{JS2}
The signature associated to the polarized variation of mixed Hodge structure of an isolated hypersurface singularity with even dimensional fibers is given by 

\begin{equation} 
\sigma=\displaystyle{\sum_{p+q=n+2}(-1)^q h_1^{pq}+2\sum_{p+q \geq n+3}(-1)^q h_1^{pq}+\sum (-1)^q h_{\ne 1}^{pq}} 
\end{equation}

\vspace{0.3cm}

\noindent
where $h_1=\dim H^n(X_{\infty})_1,h_{\ne 1}=\dim H^n(X_{\infty})_{\ne 1}$ are the corresponding Hodge numbers. This signature is $0$ when the fibers have odd dimensions.
\end{theorem}

\vspace{0.2cm}

Let $f:(\mathbb{C}^{n+1},0) \to (\mathbb{C},0)$ be a germ of analytic function having an isolated singularity at the origin. Consider

\vspace{0.3cm}
 
\begin{center}
$A_f=\displaystyle{\frac{\mathbb{C}[[x_0,...,x_n]]}{(\partial_0 f,..., \partial_n f)}}$.
\end{center}

\vspace{0.3cm}

\noindent
By Jacobson-Morosov theorem, there exists a unique increasing filtration $W_l$ on $A_f$ (or $A$) such that 

\[ \times \bar{f}:Gr_l^W A \to Gr_{l-2}^W A, \qquad \times \bar{f}^l:Gr_{l}^W A \cong Gr_{-l}^W A. \]

\vspace{0.3cm}

\noindent
Define the primitive components $P_l=PGr_l^W A_f:=\ker \bar{f}^{l+1}:Gr_{l}^W A \to Gr_{-l}^W A $. Then, we obtain a set of non-degenerate forms

\[ Q_m:PGr_m^W A \times PGr_m^W A \to \mathbb{C}. \]

\vspace{0.3cm}

\noindent
The mixed Hodge structure defined on $\Omega_f$. We defined a MHS on $\Omega_f$ in sec. 4 and saw in sections 5
and 6 that it is polarized by the form $\widehat{Res_f}$ in a way that the map $\Phi$ is an isomorphism of polarizations. 

\vspace{0.2cm}

\begin{theorem}  

The signature associated to the modified Grothendieck pairing $\widehat{Res}_{f,0}$ associated to an isolated hypersurface singularity germ $f$; is equal to the signature of the polarization form associated to the MHS of the vanishing cohomology, given by (37). 

\end{theorem}

Proof: Trivial by Theorems 5.1, 6.1 and 7.2.\\[0.2cm]

\vspace{0.5cm}

\nocite{*}
\bibliographystyle{plain}

\begin{thebibliography}{99}


\bibitem[AGV]{AGV}  V. Arnold, S. Gusein Zade, A. Varchenko;  Singularities of differentiable maps, Vol 2, Monodromy and asymptotics of integrals, 1984

\bibitem[B]{B}  E. Brieskorn Die Monodromie der isolierten Singularitäten von Hyperflächen. Man. Math. 2 (1970) 103-161

\bibitem[CKS]{CKS}  E. Cattani, A. Kaplan, W. Schmid, Degeneration of Hodge structure, Ann. of math, , second series, Vol 123, 457-535, 1986


\bibitem[CIR]{CIR}   A. Chiodo, H. Iritani, Y. Ruan,  Landau-Guinzburg/Calabi-Yau correspondence, Global mirror symmetry and Orlov equivalence, Publications mathématiques de l'IHÉS June 2014, Volume 119, Issue 1, pp 127-216 


\bibitem[D1]{D1}  P. Deligne, Theorie de Hodge. II. Publications Mathématiques de l'IHÉS (1971) 
Volume: 40, page 5-57, ISSN: 0073-8301


\bibitem[D2]{D2}  P. Deligne, Local behaviour of Hodge structures at infinity. Letter to D. Morrison; AMS/IP Studies in advanced mathematics Vol. 1, 1997

\bibitem[DW]{DW}  Donagi R., Wendland K. , From Hodge theory to integrability and TQFT tt*-geometry. proceedings of Symposia in pure mathematics, Vol 78, May 2007


\bibitem[DU]{DU}  A. Durfee, A naive guide to mixed Hodge theory, Proc. Symp. Pure Math. 40 (1983) 313-320 


\bibitem[G1]{G1}  P. A. Griffiths, Periods of integrals on algebraic manifolds. II. Local study of the period mapping. Amer. J. Math., 90: 805-865, 1968.
 
\bibitem[G2]{G2}  P. Griffiths. Hodge theory and geometry, Bull. London Math. Soc. (2004) 36 (6):  721-757. 

\bibitem[G3]{G3}  P. Griffiths , J. Harris. Principles of algebraic geometry. Wiley Classics Library. John Wiley-Sons Inc., New York, 1994.

\bibitem[G4]{G4}  P. Griffiths;  Hodge theory and representation theory,  10 Lec. at TCU, 2012

\bibitem[G5]{G5}  P. Griffiths, On the periods of certain rational integrals II, annals of math. Vol 90, no 3, 496-541  

\bibitem[H1]{H1}   C. Hertling;    Classifying spaces for polarized mixed hodge structures and Brieskorn lattices, Composithio Mathematica 116, 1-37, 1999

\bibitem[H2]{H2}  C. Hertling;     Formes bilineaires et hermitiennes pour des singularities, preprint 2003

\bibitem[H3]{H3}  C. Hertling;  Eine Torelli-Satz fur die unimodalen und bimodularen hyperflachensingularitaten, Math Ann. 302 359-394, 1995

\bibitem[HS]{HS}  C. Hertling, C. Sevenheck    Limits of families of Brieskorn lattices, and compactified classifying spaces, Advances in mathematics Vol 223, 1155-1224, 2010

\bibitem[KUL]{KUL}   V. Kulikov, Mixed Hodge structure and singularities, Cambridge University Press, 1998


\bibitem[LLS]{LLS} Li C., Li S. , Saito K., Primitive forms via polyvector fields, preprint arxiv:1311.1659v3, 2014.

\bibitem[MA]{MA}  B. Malgrange, Integrales asymptotiques et monodromie, anales Scientifique de'l I.E.N.S, 4 , tome 7, 405-430, 1974


\bibitem[P1]{P1}  G. Pearlstein; Degeneration of mixed Hodge structure, Duke math Journal Vol 110, no.2 , 2001


\bibitem[P2]{P2}  G. Pearstein, J. Fernandez; Opposite filtrations, Variation of Hodge structures and Frobenius modules, Aspects of Mathematics Volume 36, 2004, pp 19-43 

\bibitem[P3]{P3}  G. Pearlstein; Variations of mixed Hodge structure, Higgs fields, and quantum cohomology. Manuscripta Math. 102 (2000), 269-310. 

\bibitem[PH]{PH}  F. Pham,  La descente des cols par les onglets de Lefschetz, Avec Vues sur Gauss-Manin, Expose au colloque 'systemes differentiels et singularites' Astrisque, No 130, 11-47, 1985

\bibitem[PS]{PS}  Peters C. , Steenbrink J. , Mixed Hodge structures, A series of modern surveys in mathematics. Springer verlag publications, Vol 52, 2007


\bibitem[SA1]{SA1}  C. Sabbah;  Hodge theory, singularities and D-modules,  Lecture notes (CIRM, LUMINY, March 2007)

\bibitem[SA2]{SA2}  C. Sabbah, A. Douai;  Gauss-Manin systems, Brieskorn lattices and Frobenius structures II, Aspects of Mathematics Volume 36, 2004, pp 1-18 

\bibitem[SA3]{SA3}  C. Sabbah, Frobenius manifolds, isomonodromic deformations and infinitesimal period mapping, Expositiones mathematicae 16, 1998, 1-58

\bibitem[SA4]{SA4}  C. Sabbah, Vanishing cycles and hermitian duality, Tr. Mat. Inst. Steklova, 2002,  Volume 238, Pages 204–223, 2002

\bibitem[SA5]{SA5}  C. Sabbah, Examples of Frobenius manifolds, expose a venise, Juin 2005

\bibitem[SA6]{SA6}  C. Sabbah, Hypergeometric periods for tame polynomial, , Comptes Rendus de l'Académie des Sciences - Series I - Mathematics, 328:7 (1999), 603 

\bibitem[SA7]{SA7} C. Sabbah, Fourier-Laplace transform of a variation of polarized Complex Hodge structure, II, 
J. reine angew. Math. 621 (2008) 123-158, 

\bibitem[SA8]{SA8}  C. Sabbah, P. Maisonobe, Aspects of the theory of $D$-modules, Lecture Notes (Kaiserslautern 2002) Revised version July 2011, preprint


\bibitem[S1]{S1}  K. Saito;  The higher residue pairings, for a family of hypersurface singular points, Proceedings of Symposia in pure mathematics, Vol 40 (1983) Part 2 , 441-463.

\bibitem[S2]{S2}  K. Saito,   Period mapping associated to a primitive form. Publications of the research institute for math sciences Kyoto University, Volume 19, Issue 3, 1983, pp. 1231–1264 , 1983

\bibitem[SAI1]{SAI1}  M. Saito,  Mixed hodge modules and applications, Proceedings of international congress of mathematics Kyoto, Japan, Research institute for mathematical sciences, Kyoto University, 1990 

\bibitem[SAI2]{SAI2}  M. Saito,  Period mappings via Brieskorn module, Bull. soc math, France, 119, 1991, 141-171.

\bibitem[SAI3]{SAI3}  M. Saito,  Gauss-Manin systems and Mixed Hodge structure, Proc. Japan Acad., 58, Ser A, No 1, 1982

\bibitem[SAI4]{SAI4}  Saito M. , Multiplier ideals, $b$-function, and spectrum of hypersurface singularity, preprint arxiv:math/0402363v9

\bibitem[SAI5]{SAI5}  Saito. M. , Modules de Hodge Polarisables, Publ. RIMS, Kyoto Univ., 24 (1990) 2213-333

\bibitem[SAI6]{SAI6}  Saito M. , On the structure of Brieskorn lattice, Annales de l'institut Fourier, tome 39, no 1, (1989), p. 27-72

\bibitem[SC1]{SC1}   J. Scherk,   A note on two local Hodge filtrations, Proceedings of Symposia of pure mathematics, Vol 40  1983.

\bibitem[SC2]{SC2}  J. Scherk, J. Steenbrink: On the mixed Hodge structure on the cohomology of the Milnor fibre. Math. Ann. 271 (1985) 641-665

\bibitem[SCH]{SCH} W. Schmid: Variation of Hodge structure: the singularities of the period mapping. Invent, math. 22 (1973) 211-320

\bibitem[SCHU]{SCHU}     M. Schultz,  Algorithmic Gauss-Manin connection, Algorithms to compute Hodge-theoretic invariants of isolated hypersurface singularities, Ph.D dissertation, Universitat Kaiserslautern. 2002

\bibitem[JS1]{JS1}  J. Steenbrink: Limits of Hodge structures. Invent. math. 31 (1976) 229-257

\bibitem[JS2]{JS2} J. Steenbrink: Mixed Hodge structure on the vanishing cohomology. In: P. Holm (ed.): Real and complex Singularities. Oslo 1976. Sijthoff-Noordhoff, 
1977, pp. 525-563

\bibitem[JS3]{JS3} J. Steenbrink: Mixed Hodge structures associated with isolated singularities. Proc. Symp. Pure Math. 40, Part 2 (1983) 513-536


\bibitem[JS6]{JS6}   J. Steenbrink, Monodromy and weight filtrations for smoothings of isolated singularities, Composithio Mathematicae tome 97, no 1-2, 285-293, 1995
   
\bibitem[JS7]{JS7}  J. Steenbrink; Intersection form for quasi-homogeneous singularities, Composithio Mathematica Vol 34, Fasc 2, 211-223 1977

\bibitem[V]{V}   On the local residue and the intersection form on the vanishing cohomology, A. Varchenko, Math. USSR Izv. Akad. Nauk SSSR Ser. Mat., 1985, Volume 49, Issue 1, Pages 32-54,  1986



\end{thebibliography}

\end{document}